\documentclass[10pt,notitlepage,reqno]{amsart}
\usepackage[utf8]{inputenc}

\usepackage[margin=1in]{geometry}
\usepackage{mathtools} 
\usepackage{amssymb} 
\usepackage{amsthm}
\usepackage{thmtools} 
\usepackage{mathrsfs}
\usepackage{bbm}
\usepackage[shortlabels]{enumitem}
\usepackage{setspace}
\usepackage[dvipsnames]{xcolor}
\usepackage{hyperref}
\usepackage[noabbrev,capitalise]{cleveref} 
\usepackage[font=footnotesize]{caption}
\usepackage{pdflscape}

\newtheorem{theorem}{Theorem}[section] 
\newtheorem*{theorem*}{Theorem}
\newtheorem{lemma}[theorem]{Lemma}

\newtheorem{corollary}[theorem]{Corollary}

\theoremstyle{remark}
\newtheorem{remark}[theorem]{Remark}
\newtheorem{example}[theorem]{Example}
\numberwithin{equation}{section}

\usepackage{tikz}
\usetikzlibrary{decorations.markings,decorations.pathreplacing}
\usetikzlibrary{patterns}

\allowdisplaybreaks 
\newcommand{\triv}{\mathbbm{1}}
\newcommand{\F}{\mathbb{F}}
\newcommand{\N}{\mathbb{N}}
\newcommand{\Z}{\mathbb{Z}}
\newcommand{\Prob}{\mathbb{P}}
\newcommand{\cC}{\mathcal{C}}
\newcommand{\cH}{\mathcal{H}}
\newcommand{\cO}{\mathcal{O}}
\newcommand{\cP}{\mathcal{P}}
\newcommand{\cX}{\mathcal{X}}

\newcommand{\End}{\operatorname{End}}
\newcommand{\GL}{\operatorname{GL}}
\newcommand{\Irr}{\operatorname{Irr}}
\newcommand{\Syl}{\operatorname{Syl}}
\newcommand{\Sym}{\operatorname{Sym}}
\begin{document}
	
\title[Sylow double cosets of the symmetric group]{On the number and sizes of double cosets of Sylow subgroups of the symmetric group}

\author[P.~Diaconis]{Persi Diaconis}
\address[P.~Diaconis]{Department of Mathematics, Stanford University, 390 Jane Stanford Way, Stanford, 94305, California, USA}
\email{diaconis@math.stanford.edu}

\author[E.~Giannelli]{Eugenio Giannelli}
\address[E.~Giannelli]{Dipartimento di Matematica e Informatica	U.~Dini, Viale Morgagni 67/a, Firenze, Italy}
\email{eugenio.giannelli@unifi.it}

\author[R.~Guralnick]{Robert M.~Guralnick}
\address[R.~Guralnick]{Department of Mathematics, University of Southern California, Los Angeles, CA 90089-2532, USA}
\email{guralnic@usc.edu}

\author[S.~Law]{Stacey Law}
\address[S.~Law]{School of Mathematics, Watson Building, University of Birmingham, Edgbaston, Birmingham B15 2TT, UK}
\email{s.law@bham.ac.uk}

\author[G.~Navarro]{Gabriel Navarro}
\address[G.~Navarro]{Departament de Matem\`atiques, Universitat de Val\`encia, 46100 Burjassot, Val\`encia, Spain}
\email{gabriel@uv.es}

\author[B.~Sambale]{Benjamin Sambale}
\address[B.~Sambale]{Institut f\"ur Algebra, Zahlentheorie und Diskrete Mathematik, Leibniz Universit\"at Hannover, Welfengarten 1, 30167 Hannover, Germany}
\email{sambale@math.uni-hannover.de}

\author[H.~Spink]{Hunter Spink}
\address[H.~Spink]{Department of Mathematics, University of Toronto, 40 St. George St., Toronto, ON, Canada}
\email{hunter.spink@utoronto.ca}

\begin{abstract}
	Let $P_n$ be a Sylow $p$-subgroup of the symmetric group $S_n$. We investigate the number and sizes of the $P_n\setminus S_n/P_n$ double cosets, showing that `most' double cosets have maximal size when $p$ is odd, or equivalently, that $P_n\cap P_n^x=1$ for most $x\in S_n$ when $n$ is large. We also find that all possible sizes of such double cosets occur, modulo a list of small exceptions.
\end{abstract}


\maketitle


\section{Introduction}
For $n$ a natural number and $p$ a prime, let $P_n$ be a Sylow $p$-subgroup of the symmetric group $S_n$. This splits $S_n$ into a disjoint union of $P_n$ double cosets $P_n x P_n$ as $x$ varies over $S_n$. We ask:
\begin{itemize}
	\item How many double cosets are there?
	\item What are their typical sizes?
\end{itemize}

Motivation for these questions from probability, enumerative group theory and from modular representation theory are described in \cref{sec:motivation} below. Moreover, any simple question about $S_n$ is worth studying(!).

The double cosets have sizes varying between $|P_n|$ and $|P_n|^2$. 
In Corollary \ref{cor: allsizes} below, we show that if $n\geq 9$ then $S_n$ admits $(P_n,P_n)$-double cosets of all possible sizes.
Moreover, our main result shows that, for $p>2$ and $n$ large, most double cosets are as large as possible. Equivalently, since $|P_n xP_n|=|P_n|^2/|P_n\cap P_n^x|$, we have that $P_n\cap P_n^x=1$ for almost all $x$. 

\begin{theorem}\label{thm:1}
	Let $p$ be a prime. Let $f(n,p)$ be the probability that $|P_n\cap P_n^x|>1$ where $x\in S_n$ is chosen uniformly at random.
	\begin{enumerate}
		\item For $p>2$, $f(n,p)\to 0$ as $n\to \infty$, uniformly in $2<p\le n$.
		\item For $p=2$, $\liminf_{n\to\infty} f(n,2) \ge 1-e^{-1/2}$.
	\end{enumerate}
\end{theorem}

To familiarise the reader with the problem, in the following example we treat the smallest non-trivial case (i.e.~we study in detail the case of $S_p$). 
\begin{example}\label{ex:n=p}
	Let us fix the rank $n$ of our symmetric group to be equal to the prime number $p$. In this case, given $P_p\in\Syl_p(S_p)$ we have that $P_p$ is isomorphic to 
	$C_p$, the cyclic group of order $p$. Hence the double cosets have size $p$ or $p^2$. Let $n_i$ be the number of double cosets of size $p^i$, for each $i\in\{1,2\}$. Then we have
	\[ p! = pn_1+p^2n_2. \]
	If $|P_p wP_p|=|P_p|$, then $wP_p=P_pw$ so $w\in N_{S_p}(P_p)$. As explained in \cref{sec:background}, we know that $N_{S_p}(C_p)=C_p\rtimes C_{p-1}$, so $n_1=|N_{S_p}(C_p)/C_p|=p-1$ and $n_2=\frac{(p-1)!-(p-1)}{p}$.
	
	For instance, notice that if $p=11$ then $n_1=10$ and $n_2=329890$. This illustrates our main results, that all possible double coset sizes appear, and most double coset sizes are as large as possible.
\end{example}

\Cref{thm:1} shows that $p=2$ is different; letting $P_n\in\Syl_2(S_n)$, \cref{tab:1} gives some data on the number of $(P_n,P_n)$-double cosets in $S_n$ of size $2^{m+k}$, $0\le k\le m$, where $|P_n|=:2^m$ (see \eqref{eq:m}). Note that:
\begin{itemize}
	\item The largest entry in each row almost always occurs in column $2^{2m-1}$ (although the $2^{2m-1}$ and $2^{2m}$ entries are roughly equal).
	\item The first column contains all 1s as $P_n$ is self-normalising in $S_n$ when $p=2$ (see \cref{lem: minsize}).
	\item The second column (number of double cosets of size $2^{m+1}$) is explained in \cref{sec:new5} below.
	\item Going from row $n=2^k-1$ to row $n=2^k$, the entries decrease down each column. (Compare this with the structure of $P_n$; see \cref{sec:sn-sylow}.)
\end{itemize}

\begin{landscape}
	\begin{table}[p]
		\begin{small}
		\[
		\begin{array}{cc|*{17}c|c}
			\\
			\\
			\\
			\\
			\\
			&n&2^{m}&2^{m+1}&2^{m+2}&2^{m+3}&2^{m+4}&2^{m+5}&2^{m+6}&2^{m+7}&2^{m+8}&2^{m+9}&2^{m+10}&2^{m+11}&2^{m+12}&2^{m+13}&2^{m+14}&2^{m+15}&2^{m+16}&\sum\\\hline
			&1&1&&&&&&&&&&&&&&&&&1\\
			&2&1&&&&&&&&&&&&&&&&&1\\
			&3&1&1&&&&&&&&&&&&&&&&2\\
			&4&1&1&&&&&&&&&&&&&&&&2\\
			&5&1&1&1&1&&&&&&&&&&&&&&4\\
			&6&1&2&2&2&1&&&&&&&&&&&&&8\\
			&7&1&3&7&13&11&&&&&&&&&&&&&35\\
			&8&1&1&2&4&3&3&2&&&&&&&&&&&16\\
			&9&1&1&2&5&6&10&15&11&&&&&&&&&&51\\
			&10&1&1&3&8&13&22&32&43&22&&&&&&&&&145\\
			&11&1&2&4&14&39&97&218&395&342&&&&&&&&&1112\\
			&12&1&3&8&17&27&53&97&154&247&341&197&&&&&&&1145\\
			&13&1&3&9&23&53&150&399&965&2173&3818&3335&&&&&&&10929\\
			&14&1&4&15&50&135&341&826&1942&4399&8983&13737&10967&&&&&&41400\\
			&15&1&5&22&89&328&1202&4268&13960&41210&104946&194791&181963&&&&&&542785\\
			&16&1&1&2&6&15&24&55&100&209&407&955&1938&4755&8390&13783&9743&&40384\\
			&17&1&1&2&6&16&29&77&189&537&1609&5223&15898&45965&113336&208574&191706&&583169\\
			&18&1&1&2&7&21&51&158&442&1240&3555&10602&32233&95157&257733&589685&974086&816834&2781808
		\end{array}
		\]
		\caption{The number of $(P_n,P_n)$-double cosets in $S_n$ according to their size, where $P_n$ is a Sylow 2-subgroup of $S_n$. Note that the size of any such double coset must be a power of 2 between $2^m$ and $2^{2m}$; see \eqref{eq:dc-size-range}. Here, $m$ is given by $|P_n|=2^m$; see \eqref{eq:m} for explicit formulas for the value of $m$.}
		\label{tab:1}
		\end{small}
	\end{table}
\end{landscape}

\subsection{Motivation}\label{sec:motivation}

Our route to studying these problems comes from `P\'olya theory' -- enumeration under symmetry. Let $\cX$ be a finite set and $G$ a finite group acting on $\cX$.
This splits $\cX$ into disjoint orbits 
\[ \cX=\cO_1\cup\cO_2\cup\cdots\cup\cO_k. \]
Natural questions are:
\begin{itemize}
	\item How many orbits are there?
	\item What are the typical sizes?
	\item Do the orbits have `nice names'?
	\item How can one `pick a random orbit'?
\end{itemize}

Of course, in this generality, this is a hopelessly out of focus question; there are too many groups acting on too many sets. Nonetheless, there are many important special cases. See \cite{Keller} for a review when $\cX$ is a group.

Computer scientists Jerrum and Goldberg \cite{Jerrum,GJ} introduced a general algorithm for random generation which allows the first two problems to be studied -- the Burnside process. Their interest was computational complexity and they highlight special examples where the questions are \#P-complete.

In contrast, \cite{BD,DHowes,DZ20,DZ25} show that there are many examples where enumeration is feasible (and interesting):
\begin{itemize}
	\item Suppose $\cX=G$ and $G$ acts on itself by conjugation ($x^g=g^{-1}xg$). Then the orbits are conjugacy classes and the questions become: how many classes? What are their typical sizes?
	
	When $G=S_n$, the conjugacy classes are indexed by partitions and the Burnside process gives a useful algorithm for generating a random partition of $n$. Work in \cite{DHowes} shows this is effective for $n$ up to $10^9$, for instance.
	
	\item Let $H$ and $K$ be subgroups of a finite group $G$. Then $H\times K$ acts on $G$ via $(h,k)\cdot g=hgk^{-1}$, with orbits the double cosets $H\setminus G/K$. 
	
	When $G=S_n$ and $H=S_\lambda$, $K=S_\mu$ for two partitions $\lambda$ and $\mu$ of $n$ (i.e~$H$ and $K$ are parabolic subgroups), the double cosets are indexed by contingency tables: arrays of non-negative integers with row sums $\lambda$ and column sums $\mu$ \cite[Theorem 1.3.10]{JK}. For references, enumerative theory and statistical applications, see \cite{DHowes,DSimper}.
	
	If $G=\GL_n(\F_q)$ and $H=K=B$, the Borel subgroup of invertible upper triangular matrices, then we have the Bruhat decomposition $\GL_n(\F_q)=\bigsqcup_{w\in S_n} BwB$; see \cite{DRS} for probabilistic applications.
\end{itemize}

Michael Geline asked about $P_n\setminus S_n/P_n$. This connects to modular representation theory by the following route. Consider a finite group $G$ with a split $BN$-pair $B,N,U$ in characteristic $p$ (see \cite{Carter72} or \cite{Carter85} for definitions and references). Let $k$ be an algebraically closed field of characteristic $p$. The irreducible representations of $G$ over $k$ can be studied via the approach of \cite{Sawada,Curtis70,Richen,Green} and \cite{Tinberg} (this last reference is a well written account with full references). A central piece of the story is the Hecke algebra
\[ E = L_k(U\setminus G/U) = \End_G((\triv_U)^G). \]
They show:
\begin{itemize}
	\item $E$ is a Frobenius algebra.
	\item Every simple right $E$--module is 1-dimensional, given by a multiplicative character $\psi:E\to k$.
	\item Each such $\psi$ is determined by a vector of parameters $(\chi,u_1,u_2,\dotsc,u_m)$ with $\chi$ a linear character of the Borel subgroup $B$ and $u_i\in k$.
	\item There is a bijective correspondence between the set of irreducible $kG$--modules and the set of such characters $\psi$.
\end{itemize}

One hope for studying the difficult problem of understanding representations of $S_n$ over $k$ \cite{Kleshchev} is to study the Hecke algebra $\cH_n(k):=L_k(P_n\setminus S_n/P_n)$. Understanding the number of double cosets (i.e. the $k$-dimension of $\cH_n(k)$) and their sizes seems like a natural first step. 
It is worth mentioning that the representation theory of the algebra $\cH_n(k)$ is closely related to the decomposition  of the permutation character $(\triv_{P_n})^{S_n}$ into irreducible constituents. (We refer the reader to \cite[Chapter 11D]{CR81} for the complete definition and properties of this correspondence.) Exploiting this connection, the exact number of irreducible representations of $\cH_n(k)$ has been computed in \cite[Corollary B]{GL1} for any field $k$ of odd characteristic.

Alas, our main results show that most double cosets have the same size, so that `size' does not usefully distinguish them. Still, it does give a good hold on dimension.

\subsection{Outline}

\Cref{sec:background} gives background on $P_n$ and double coset enumeration. Furthermore, we prove that all possible sizes occur for $n\ge 9$. \Cref{sec:abelian} studies a special case when $n=kp$ and $1\le k\le p-1$. Then, sharp formulas and asymptotics are available. It may be read now for further motivation. \Cref{sec:4} deduces \cref{thm:1}. Our proof uses a result on random generation of $A_n$ due to Eberhard--Garzoni \cite{EG21,EG22}. \Cref{sec:new5} develops a complete understanding of double cosets of size $p|P_n|$ for all $p$. It also contains useful facts for size $p^k|P_n|$ for general $k$. The final section contains remarks and open problems.

\medskip

\subsubsection*{Acknowledgements}
We thank Michael Geline, Michael Howes, Marty Isaacs, Markus Lincklemann, Radha Kessar, John Murray and Richard Stanley for many interesting comments and helpful discussions. P.D. acknowledges support from NSF grant 1954042. E.G. acknowledges support from the European Union - Next Generation EU, M4C1, CUP B53D23009410006, PRIN 2022.
\bigskip
\section{Background}\label{sec:background}
\subsection{Sylow subgroups of $S_n$}\label{sec:sn-sylow}
The Sylow $p$-subgroups of $S_n$ were first determined in \cite{Kaloujnine}. This is also sometimes attributed to Cauchy, see \cite{Meo}. The facts below are standard, see \cite{OlssonBook} or \cite{JK}; see also \cite{Wildon} for an alternative description.

The largest power of $p$ dividing $n!$ is $p^m$ where $m$ is given as follows:
\begin{equation}\label{eq:m}
	m=\sum_{a=1}^\infty\left\lfloor\frac{n}{p^a}\right\rfloor = \frac{n-D_p(n)}{p-1}.
\end{equation}
Here, if $n=\sum_{i=0}^\infty a_ip^i$ is the $p$-adic expansion of $n$, i.e.~$a_i\in\{0,1,\dotsc,p-1\}$ for all $i$, then $D_p(n)=\sum_{i=0}^\infty a_i$ is the sum of the digits (this was known to Legendre; see \cite{Leg}, whose first edition was published in 1798).

For example, if $n=p^2$ then $m=\frac{p^2-1}{p-1}=p+1$, so $|P_n|=p^{p+1}$. For this case, $P_n$ may be pictured as $C_p^p\rtimes C_p\cong C_p\wr C_p$. When $p=3$, we have the following diagram:
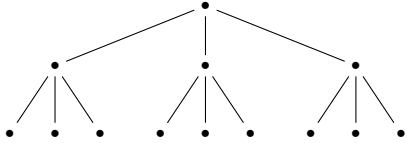
\begin{center}
	\begin{tikzpicture}[scale=1.0, every node/.style={scale=0.7}]
		\draw (0,0.2) node(R) {$\bullet$};
		\draw (-2,-0.6) node(1) {$\bullet$};
		\draw (0,-0.6) node(2) {$\bullet$};
		\draw (2,-0.6) node(3) {$\bullet$};
		\draw (R) -- (1);
		\draw (R) -- (2);
		\draw (R) -- (3);
		\draw (-2.6,-1.5) node(11) {$\bullet$};
		\draw (-2,-1.5) node(12) {$\bullet$};
		\draw (-1.4,-1.5) node(13) {$\bullet$};
		\draw (1) -- (11);
		\draw (1) -- (12);
		\draw (1) -- (13);
		\draw (-0.6,-1.5) node(21) {$\bullet$};
		\draw (0,-1.5) node(22) {$\bullet$};
		\draw (0.6,-1.5) node(23) {$\bullet$};
		\draw (2) -- (21);
		\draw (2) -- (22);
		\draw (2) -- (23);
		\draw (1.4,-1.5) node(31) {$\bullet$};
		\draw (2,-1.5) node(32) {$\bullet$};
		\draw (2.6,-1.5) node(33) {$\bullet$};
		\draw (3) -- (31);
		\draw (3) -- (32);
		\draw (3) -- (33);
	\end{tikzpicture}
\end{center}

\noindent Here, each of the $p$ copies of $C_p$ in the base group acts cyclically on each set of leaves, and the wreathing $C_p$ permutes the $p$ branches from the root cyclically. Iterating this construction, for $n=p^k$ we have that
\begin{equation}
	P_n = \underbrace{C_p \wr C_p \wr \cdots \wr C_p}_{k\text{ times}},\qquad |P_n|=p^{1+p+\cdots+p^{k-1}}.
\end{equation}
For general $n$, write $n=\sum_{j=0}^\infty a_jp^j$ where $0\le a_j\le p-1$ for each $j$. Then $P_n$ is isomorphic to the direct product of $a_j$ copies of the Sylow $p$-subgroup of $S_{p^j}$ taken over all $j\ge 0$. A useful fact, needed below, is
\begin{equation}\label{eq:normaliser}
	N_{S_n}(P_n) \cong \prod_j N_{S_{p^j}}(P_{p^j}) \wr S_{a_j}
\end{equation}
where $N_{S_{p^j}}(P_{p^j}) =P_{p^j}\rtimes (C_{p-1})^j$. A careful version of the isomorphism and further details can be found in \cite[Section 2B]{Giannelli} where it is applied to give explicit McKay bijections for $S_n$ and $A_n$. 

\subsection{Double cosets for $H\times K$ acting on $G$}
Let $G$ be a finite group and $H$ and $K$ be subgroups of $G$. The set of orbits of $H\times K$ acting on $G$ by $(h,k)\cdot g = hgk^{-1}$ is denoted $H\setminus G/K$, the set of $(H,K)$-double cosets of $G$. For a textbook treatment, see \cite[p.23]{Suzuki}. The Orbit--Stabiliser theorem implies
\begin{equation}\label{eq:dc-size}
	|HxK| = \frac{|H||K|}{|H\cap xKx^{-1}|}
\end{equation}
for $x\in G$. Here are three formulas for the number of double cosets: the first is an easy application of Orbit--Stabiliser; see for example \cite[Ex.~7.77]{Stanley} and \cite[p.44]{Curtis99} for the latter two.
\begin{subequations}
	\begin{align}
		|H\setminus G/K| &= \frac{1}{|H||K|} \sum_{h\in H,k\in K} |G_{hk}|,\quad\text{where } G_{hk}=\{g\mid hgk^{-1}=g\}, \\
		|H\setminus G/K| &= \sum_{\chi\in\Irr(G)} a(\chi)b(\chi) = \langle (\triv_H)^G,(\triv_K)^G \rangle \nonumber\\
		& \qquad\text{where }(\triv_H)^G=\sum_{\chi\in\Irr(G)}a(\chi)\chi,\ (\triv_K)^G=\sum_{\chi\in\Irr(G)}b(\chi), \label{eq:part-b}\\
		|H\setminus G/K| &= \frac{|G|}{|H||K|}\sum_{\cC}\frac{|\cC\cap H||\cC\cap K|}{|\cC|}\label{eq:part-c}
	\end{align}
	where $\cC$ runs over the conjugacy classes of $G$.
\end{subequations}
See \cite{Renteln} for applications of \eqref{eq:part-c} to $|P_n\setminus S_n/P_n|$: this paper contains a useful review for Sylow $p$-subgroups of $S_n$ with full proofs.

\begin{remark}
	In the case of $p=2$, the sequence of numbers $|P_n\setminus S_n/P_n|$ is recorded on the Online Encyclopedia of Integer Sequences as A360808 \cite{OEIS}, and the formula given there follows from \eqref{eq:part-b}. 
	
	To see this: $|P_n\setminus S_n/P_n|$ is calculated as $\langle U_n,U_n\rangle$ where Stanley defines the symmetric functions $U_n$ and $T_k$ as follows. Set $U_n = T_{a_0}T_{a_1}...T_{a_s}$, where $T_k$ is recursively defined by $T_0=p_1$ (power sum) and $T_k=h_2[T_{k-1}]$ (plethysm) for $k>0$. 
	Under the Frobenius characteristic isomorphism, $U_n$ corresponds to the symmetric group character $(\triv_{P_n})^{S_n}$, since $T_k$ corresponds to $(\triv_{S_2\wr S_2\wr \cdots \wr S_2})^{S_{2^k}}$. In this latter expression the trivial character is induced from a $k$-fold wreath product of cyclic groups of order 2, and such a group is exactly a Sylow 2-subgroup of $S_{2^k}$.
\end{remark}

All of these formulas were useful in our preliminary work, collecting examples. We have not seen how to use any of them to derive results for general $n$.

Despite these nice formulas, it is worth remembering that there are many examples where computing $|HxK|$ or $|H\setminus G/K|$ exactly is \#P-complete. See \cite{DSimper} or \cite{DMalliaris} for references and examples.

\subsection{Smallest and largest sizes}
Let $G$ be a finite group and $x\in G$. Take $H=K\le G$ in this section. From \eqref{eq:dc-size} we immediately observe that
\begin{equation}\label{eq:dc-size-range}
	|H|\le|HxH|\le |H|^2.
\end{equation}
Clearly the lower bound is attained at $x=1$. More generally, we observe that $|HxH|=|H|$ if and only if $xH=Hx$, that is, if and only if $x\in N_G(H)$.

\begin{lemma}\label{lem: minsize}
	For finite groups $H\le G$, the number of double cosets $HxH$ of size $|H|$ is $|N_G(H):H|$ and such double cosets are naturally labelled by $N_G(H)/H$.
\end{lemma}

\begin{remark}
	Lemma \ref{lem: minsize} shows that the set $N_G(H)/H$ naturally labels those double cosets in $H\setminus G/H$ of minimal size.  
	Despite this, it seems to be extremely difficult to find a labelling for all double cosets. For instance, 
	even for $P_p\le S_p$, we do not have a natural labelling of those double cosets of size $p^2$.
	When $p=2$, \eqref{eq:normaliser} shows that $|N_{S_n}(P_n):P_n|=1$, explaining the first column of \cref{tab:1}.
\end{remark}

\smallskip

On the other end of the spectrum, double cosets of maximal size need not always exist in general. For instance, if $H\trianglelefteq G$ then all double cosets are simply one-sided cosets, of minimal size $|H|$.
Nevertheless, the following result by Zenkov and Mazurov in \cite{ZM} settles the case for $P_n\setminus S_n/P_n$.

\begin{theorem}{\cite[Theorem 1]{ZM}}\label{thm:sylow-trivial-intersection}
	For $p$ a prime and $n$ a natural number, the symmetric group $S_n$ contains at least two Sylow $p$-subgroups with trivial intersection if and only if 
	\[ (p,n)\notin\{(3,3),(2,2),(2,4),(2,8)\}. \]
\end{theorem}

\begin{remark}
	Another way to see this when $p\notin\{2,3\}$ is to use a theorem of Granville--Ono \cite{GO}, which asserts that for $p\notin\{2,3\}$ and all $n$, $S_n$ admits $p$-blocks with trivial defect group. This implies the existence of $P\in\Syl_p(S_n)$ and $x\in S_n$ such that $P\cap P^x=1$, using a theorem of Green \cite[Corollary 4.21]{NavarroBook}.
	
	Indeed, following a suggestion of Radha Kessar, by considering $p$-defect zero characters we can also obtain a first bound on the number of Sylow-$p$ double cosets of $S_n$ of maximal size. Let $X$ denote the set of irreducible characters of $S_n$ of $p$-defect 0 (under the natural bijection between $\Irr(S_n)$ and the set $\cP(n)$ of partitions of $n$, such characters are labelled by $p$-core partitions). Then the number of $(P_n,P_n)$-double cosets of $S_n$ of size $|P_n|^2$ is at least $\sum_{\chi\in X} (\chi(1)_{p'})^2$, where if $m\in\N$ then $m_{p'}$ denotes its $p'$-part. This follows from the fact that a defect zero block is projective as a $kP_n$-$kP_n$-bimodule.
\end{remark}

The next section treats double cosets of intermediate sizes. As mentioned in the introduction, our first main result (\cref{thm:all-sizes}) extends Theorem \ref{thm:sylow-trivial-intersection} by showing that $P_n\setminus S_n/P_n$ admits double cosets of any possible size (modulo a few exceptions when $n<9$).

\subsection{Intermediate sizes}
Let $n$ be a natural number and $p$ a prime. Let $P$ denote a Sylow $p$-subgroup of $S_n$ and $x\in S_n$.
Since $|PxP|=|P|/|P:P\cap xPx^{-1}|$ from \eqref{eq:dc-size}, it is easy to see that $|PxP|$ must be a $p$-power between $|P|$ and $|P|^2$. Moreover, it turns out that all possible sizes can occur for such Sylow-$p$ double cosets.
In order to prove this statement, we first need the following technical lemma. 

\begin{lemma}\label{lem:sylow-2}
	Let $k\in\N$ with $k\ge 2$ and $n=2^k+1$. Then there exist $P,Q\in\Syl_2(S_n)$ such that $|P\cap Q|=2$ and $P$ and $Q$ have no common fixed point. 
\end{lemma}

\begin{proof}
	The case $k\le 4$ can be checked using \cite{GAP48}. Now assume $k\ge 5$. Let $n_1=2^{k-1}$, $Y_1=\Sym(\{1,\ldots,n_1\})$ and $n_2=2^{k-1}+1$, $Y_2=\Sym(\{n_1+1,\ldots,n\})\cong S_{n_2}$. By Theorem \ref{thm:sylow-trivial-intersection}, there exist $P_1,Q_1\in\Syl_2(Y_1)$ such that $P_1\cap Q_1=1$. By induction there exist $P_2,Q_2\in\Syl_2(Y_2)$ such that $|P_2\cap Q_2|=2$ and $P_2$ and $Q_2$ have no common fixed point. Observe that $P_i\cong Q_i$. Take involutions $x,y\in S_n$ such that $P_2=P_1^x$ and $Q_2=Q_1^y$. The claim holds for $P=P_1P_2\langle x\rangle$ and $Q=Q_1Q_2\langle y\rangle$. 
\end{proof}

\begin{theorem}\label{thm:all-sizes}
	Let $n\in\N$ and $p$ be a prime. Let $p^m$ be the $p$-part of $n!$. Then for every $k\in\{0,1,\dotsc,m\}$ there exist $P,Q\in\Syl_p(S_n)$ such that $|P\cap Q|=p^k$, if and only if
	\[ (n,p,k)\notin\{(2,2,0),(4,2,0),(4,2,1),(8,2,0),(3,3,0),(6,3,1) \}. \]
\end{theorem}

\begin{proof}
	The case $n\le 10$ can be checked with \cite{GAP48}. Now assume $n\ge 11$.
	The case $k=0$ follows from Theorem \ref{thm:sylow-trivial-intersection}. For $k=m$ we can choose $P=Q$. Thus, let $1\le k<m$ and $P\in\Syl_p(S_n)$. We argue by induction on $n$.  
	
	\textbf{Case~1:} Let us first assume that $n=p^\ell$ is a prime power.\\ 
	Let $m(\ell)\in\N$ be such that $|P|=p^{m(\ell)}$. 
	The description of the algebraic structure of the Sylow $p$-subgroups of symmetric groups implies that 
	$m(\ell)=p\cdot m(\ell-1)+1$, for any $\ell\geq 2$.
	Since $0<k<m(\ell)$, we have $\ell\geq 2$. 
	We partition $\{1,\ldots,n\}=N_1\cup\ldots\cup N_p$ where $N_i:=\{n(i-1)/p+1,\ldots,ni/p\}$ for $i=1,\ldots,p$. Let $Y_i:=\Sym(N_i)$ and $P_i\in\Syl_p(Y_i)$ for each $i$. 
	Let $x\in S_n$ be an element of order $p$ which permutes the $P_i$ cyclically, e.g.~$P_{i+1}=P_i^x$ for $i=1,\ldots,p-1$. 
	Then $P:=P_1\ldots P_p\langle x\rangle\cong P_1\wr C_p$ is a Sylow $p$-subgroup of $S_n$. 
	Suppose first that $n\ne 16$. Since $n\ge 11$ and $k<m(\ell)-1=p\cdot m(\ell-1)$ we can choose inductively $Q_i\in\Syl_p(Y_i)$ such that 
	\[ |(P_1\cap Q_1)\times\ldots\times(P_p\cap Q_p)|=p^k\ \ \text{and such that}\ \ |P_1\cap Q_1|\neq |P_2\cap Q_2|. \]
	Let $Q\in\Syl_p(S_n)$ be such that $Q_1\times\cdots\times Q_p\leq Q$. For any $y\in P\cap Q$ we have that $\langle y\rangle$ acts on $\{P_1,\ldots, P_p\}$ and on $\{Q_1,\ldots, Q_p\}$. Since $\{N_1,\ldots, N_p\}$ is a system of imprimitivity for both $P$ and $Q$ we have that $P_i^y=P_j$ if and only if $Q_i^y=Q_j$. It follows that $\langle y\rangle$ acts on the set $\Omega:=\{P_1\cap Q_1, P_2\cap Q_2, \ldots, P_p\cap Q_p\}$. Since $|P_1\cap Q_1|\neq |P_2\cap Q_2|$ the action is not transitive. We deduce that $\langle y\rangle$ acts trivially on $\Omega$ and hence that $P_i^y=P_i$ and 
	$Q_i^y=Q_i$. This shows that $y\in Y_1\times\cdots\times Y_p$ and therefore that $P\cap Q=(P_1\cap Q_1)\times\ldots\times(P_p\cap Q_p)$. We conclude that $|P\cap Q|=p^k$, as desired. 
	
	Next suppose $k=m(\ell)-1$. Here we choose a different partition $\{1,\ldots,n\}=N_1\cup\ldots\cup N_{n/p}$ where $N_i:=\{(i-1)p+1,\ldots,ip\}$ for $i=1,\ldots,n/p$. 
	Let $Y_i=\Sym(N_i)\cong S_p$ and $P_i\in\Syl_p(Y_i)$. 
	We let $B:=P_1\times P_2\times\cdots\times P_{n/p}$ and $S\cong S_{n/p}$ be a fixed subgroup of $N_{S_n}(B)$ permuting the sets $N_1,\ldots, N_{n/p}$. 
	Since $11\le n\ne 16$, using the inductive hypothesis we can find $T_1,T_2\in\Syl_p(S)$ such that $|T_1:T_1\cap T_2|=p$. Now let $P=B\rtimes T_1$ and $Q=B\rtimes T_2$. It is clear that $P,Q\in\mathrm{Syl}_p(S_n)$.  
	Setting $Z:=B(T_1\cap T_2)$ we notice that $Z\leq P\cap Q\leq P$ and that $|P:Z|=|T_1:T_1\cap T_2|=p$. It follows that $Z=P\cap Q$. In fact, assuming for a contradiction that $Z\neq P\cap Q$, we would have $P\cap Q=P$ and hence that $T_1\leq Q\cap S=BT_2\cap S=T_2$ which is a contradiction. 
	We conclude that $|P\cap Q|=|Z|=|B(T_1\cap T_2)|=p^{m(\ell)-1}$, as desired. 
	
	Finally, to conclude Case 1 we consider $n=16$ and $p=2$. Here we can argue similarly unless $k=1$. Fortunately, it turns out that most Sylow intersections are small. We have verified this special case by choosing random Sylow $2$-subgroups by computer.
	
	\textbf{Case~2:} $n$ is not a $p$-power.\\
	Let $n_1<n$ be the largest $p$-power $\le n$. Let $N_1:=\{1,\ldots,n_1\}$ and $N_2:=\{n_1+1,\ldots,n\}$. 
	Let $Y_1:=\Sym(N_1)$ and $Y_2:=\Sym(N_2)$. 
	Then we may assume that $P=P_1\times P_2$ where $P_1\in\Syl_p(Y_1)$ and $P_2\in\Syl_p(Y_2)$. If $p\ge 5$, we can choose $Q_1\in\Syl_p(Y_1)$ and $Q_2\in\Syl_p(Y_2)$ inductively such that $Q:=Q_1\times Q_2\in\Syl_p(S_n)$ with
	\[|P\cap Q|=|P_1\cap Q_1||P_2\cap Q_2|=p^k.\] 
	For $p=3$, we can argue as above except the case $n-n_1\in\{3,6\}$ with $k=1$ needs attention. However, here we can choose $Q_1\in\Syl_3(Y_1)$ such that $|P_1\cap Q_1|=1$, because $n_1\ge 9$. Hence, we may now assume that $p=2$.
	Only the cases $k\in\{1,2\}$ are problematic. If $n=12$, we can find $P,Q$ by computer random generation. Consequently, we may assume that $n\ge 17$. Now only the case $n-n_1=4$ and $k=1$ is left. 
	Let $N_1:=\{1,\ldots,n_1+1\}$. By \cref{lem:sylow-2}, there exist $P_1,Q_1\in\Syl_2(\Sym(N_1))$ such that $|P_1\cap Q_1|=2$ and $P_1$ and $Q_1$ have no common fixed point. Without loss of generality, we may assume that $P_1$ fixes $n_1+1$ and $Q_1$ fixes $n_1$. Let $N_2:=\{n_1+1,\ldots,n\}$ and $N_2':=\{n_1,n_1+2,n_1+3,n\}$. It is easy to find $P_2\in\Syl_2(\Sym(N_2))$ and $Q_2\in\Syl_2(\Sym(N_2'))$ with $P_2\cap Q_2=1$. Now $P=P_1\times P_2$ and $Q=Q_1\times Q_2$ are Sylow $2$-subgroups of $S_n$ with $|P\cap Q|=2$. 
\end{proof}

\begin{corollary}\label{cor: allsizes}
	For $n\ge 9$, all possible Sylow-$p$ double coset sizes occur in $S_n$.
\end{corollary}

In \cref{sec:new5}, we further investigate the double cosets of second smallest possible size.

\bigskip
\section{Abelian Sylow $p$-subgroups}\label{sec:abelian}

For $n=kp$ with $k\in\{1,2,\dotsc,p-1\}$, \cref{sec:sn-sylow} shows that $P_n\cong(C_p)^k$, where, as before, $P_n$ denotes a Sylow $p$-subgroup of $S_n$. Now $p^k\le|P_nxP_n|\le p^{2k}$, and the arguments below show that for $p>2$:
\begin{itemize}
	\item All values $p^a$ with $k\le a\le 2k$ occur as sizes of $(P_n,P_n)$-double cosets.
	\item Almost all double cosets are of size $p^{2k}$, so the total number of double cosets is asymptotically $\frac{(kp)!}{p^{2k}}$.
	\item For $p$ large, the number $n_a$ of double cosets of size $p^a$ is super-exponentially increasing from $n_k=\frac{(p(p-1)^k)\cdot k!}{(pk)!}$ to $n_{2k}=p^{2k}$.
\end{itemize}

\Cref{sec:exact-formulas} gives exact formulas, and \cref{sec:approx} gives useful approximations.

\subsection{Exact formulas}\label{sec:exact-formulas}
Since the case of $k=1$ was discussed in \cref{ex:n=p}, we may assume $k\ge 2$ in the following.
\begin{theorem}
	For a prime $p$ and $n=kp$ where $2\le k\le p-1$, let $n_a$ be the number of Sylow-$p$ double cosets of $S_n$ of size $p^a$, for each $k\le a\le 2k$. Then
	\begin{equation}\label{eq:na}
		n_a = \frac{1}{p^a} \sum_{j=2k-a}^k \big((k-j)p\big)!j!\binom{k}{j}^2\big(p(p-1)^j\big)(-1)^{j-(2k-a)}\binom{j}{2k-a}.
	\end{equation}
\end{theorem}

\begin{proof}
	The result will follow from considering the following generating function. Let
	\[ f_{k,p}(x) = \sum_{i=0}^k \#\{\pi\in S_{kp} \mid |C_p^k\cap\pi^{-1} C_p^k\pi| = p^i\}x^i. \]
	We claim that
	\begin{equation}\label{eq:gen-fn}
		f_{k,p}(x) = \sum_{i=0}^k\big((k-i)p)\big)!i!\binom{k}{i}^2\big(p(p-1)(x-1)\big)^i.
	\end{equation}
	First, to prove \eqref{eq:na} from \eqref{eq:gen-fn}: note that the $x^{2k-a}$ coefficient of $f_{k,p}$ is the number of $\pi\in S_{kp}$ such that $|C_p^k\pi C_p^k|=p^a$.
	If $|C_p^k\pi C_p^k|=p^a$, then there are exactly $p^a$ elements $\sigma\in S_{kp}$ such that $C_p^k\pi C_p^k=C_p^k\sigma C_p^k$ (namely, the set of such $\sigma$ is the double coset $C_p^k\pi C_p^k$ itself). Hence
	\[ n_a = \frac{1}{p^a}[x^{2k-a}]f_{k,p}(x) = \frac{1}{p^a}\sum_{j=2k-a}^k \big((k-j)p\big)!j!\binom{k}{j}^2\big(p(p-1)\big)^j(-1)^{j-(2k-a)}\binom{j}{2k-a}. \]
	To conclude, we prove that \eqref{eq:gen-fn} holds. Fix $k$ and $p$ and write $f_{k,p}(x) = \sum_i a_ix^i$. Let $\sigma_i$ be the cycle $\big( (i-1)p+1, \dotsc, ip\big)$. Then $a_i$ counts the number of $\pi\in S_{kp}$ such that the subset $R\subset\{1,2,\dotsc,k\}$, of indices $r$ which have the property that there exists another index $s$ with $\pi\sigma_r\pi^{-1}\in\langle\sigma_s\rangle$, has exactly size $i$. (Here $r,s\in\{1,2,\dotsc,k\}$.)
	
	Then the $x^i$ coefficient of $f_{k,p}(x+1)$ is $a_i+a_{i+1}\binom{i+1}{i} + \cdots + a_k\binom{k}{i}$ and we want to show that
	\begin{equation}\label{eq:rtp}
		a_i+a_{i+1}\binom{i+1}{i} + \cdots + a_k\binom{k}{i} = \big((k-i)p\big)! i! \binom{k}{i}^2 \big(p(p-1)\big)^i.
	\end{equation}
	The right hand side of \eqref{eq:rtp} counts the number of $\pi\in S_{kp}$ together with a distinguished subset $A\subset\{1,\dotsc,k\}$ of size $i$ such that for each $a\in A$ there exists $b\in\{1,\dotsc,k\}$ satisfying $\pi\sigma_a\pi^{-1}\in\langle \sigma_b\rangle$. Indeed, concretely choose:
	\begin{itemize}
		\item a pair of subsets $A,B\subset\{1,\dotsc,k\}$ with $|A|=|B|=i$ -- there are $\binom{k}{i}^2$ such choices;
		\item a bijection $\phi:A\to B$ -- there are $i!$ such choices;
		\item for each $a\in A$ (noting that $|A|=i$), a bijection $r_a:\{(a-1)p+1,\dotsc,ap\}\to\{(\phi(a)-1)p+1,\dotsc,\phi(a)p\}$ such that $r_a\sigma_ir_a^{-1}\in\langle\sigma_{\phi(a)}\rangle$ -- there are $p(p-1)$ such choices, for instance by first choosing $r_a((a-1)p+1)$ from one of $p$ possible values, then by choosing $r_a((a-1)p+2)$ from one of $p-1$ possible values, which then determines $r_a$; and
		\item a bijection $\tau:\{1,\dotsc,kp\}\setminus \bigcup_{a\in A}\{(a-1)p+1,\dotsc,ap\} \to \{1,\dotsc,kp\}\setminus\bigcup_{b\in B}\{(b-1)p+1,\dotsc,bp\}$ -- notice the domain and codomain of $\tau$ each have size $(k-i)p$ so there are $((k-i)p)!$ such choices;
	\end{itemize}
	then take $\pi$ to be defined by the $\{r_a\}_{a\in A}$ and $\tau$, whose domains and codomains both have disjoint unions $\{1,\dotsc,kp\}$.
	
	On the other hand, the left hand side of \eqref{eq:rtp} creates the pair $(\pi,A)$ as follows: first choose $j$ such that $i\le j\le k$, then take a permutation $\pi$ counted by $a_j$, that is, $\pi$ satisfies $|R|=j$ where 
	\[ R:= \{r\in\{1,\dotsc,k\} \mid \exists\ s,\ \pi\sigma_r\pi^{-1}\in\langle\sigma_s\rangle \}, \]
	and then choose the distinguished set $A$ to be one of the $\binom{j}{i}$ $i$-element subsets $A$ of the $j$-element subset $R\subset\{1,\dotsc,k\}$.
	Thus \eqref{eq:rtp} holds as desired, and this concludes the proof.
\end{proof}

\begin{example}
	The first few cases of the generating function considered above are, for example,
	\begin{itemize}
		\item $f_{1,p}(x) = p!+p(p-1)(x-1)$,
		\item $f_{2,p}(x) = (2p)! + p!\cdot 4p(p-1)(x-1)+2p^2(p-1)^2(x-1)^2$,
		\item $f_{3,p}(x) = (3p)! + (2p)!\cdot 9p(p-1)(x-1) + p!\cdot 18p^2(p-1)^2(x-1)^2 + 6p^3(p-1)^3(x-1)^3$.
	\end{itemize}
\end{example}

\subsection{Some approximations}\label{sec:approx}

A glance at the formula \eqref{eq:na} shows it is not so easy to understand; an alternating sum with factorials and binomial coefficients. The following result gives sharp upper and lower bounds. They show, in a strong sense, that when $p$ is large, most double cosets have size $p^{2k}$, uniformly in $k$. This implies that the number of double cosets is asymptotic to $\frac{(kp)!}{p^{2k}}$.

\begin{theorem}
	With notation as above, for all $p$ and $1\le k\le p-1$ we have that
	\[ \frac{(kp)!}{p^{2k}} \left(1-\frac{1}{(p-2)!}\right) \le n_{2k} \le \frac{(kp)!}{p^{2k}}. \]
\end{theorem}

\begin{proof}
	From \eqref{eq:na},
	\begin{equation}\label{eq:n_2k}
		n_{2k} = \frac{1}{p^{2k}} \sum_{j=0}^k (-1)^j\big((k-j)p\big)!j!\binom{k}{j}^2\big(p(p-1)\big)^j.
	\end{equation}
	Note that the $j=0$ term in the sum is $(kp)!$. The proof proceeds by showing that the other terms are super exponentially smaller than the $j=0$ term. For $0\le j\le k-1$ let
	\[ \Gamma_j = \big((k-j)p\big)!j!\binom{k}{j}^2\big(p(p-1)\big)^j, \]
	so $n_{2k}=\frac{1}{p^{2k}}\sum_{j=0}^k (-1)^j\Gamma_j$.
	We compute
	\begin{align*}
		\frac{\Gamma_{j+1}}{\Gamma_j} &= \frac{\big((k-j)p-p\big)!(j+1)!\binom{k}{j+1}^2\big(p(p-1)\big)^{j+1}}{\big((k-j)p\big)!j!\binom{k}{j}^2\big(p(p-1)\big)^j}\\
		&= \frac{(j+1)p(p-1)\left(\tfrac{k-j}{j+1}\right)^2}{\big((k-j)p\big)\big((k-j)p-1\big)\cdots\big((k-j)p-p+1\big)}\\
		&= \frac{p(p-1)(k-j)^2}{(j+1)\big((k-j)p\big)\big((k-j)p-1\big)\cdots\big((k-j)p-p+1\big)}.
	\end{align*}
	Since $(p-1)(k-j)\le(k-j)p-1$, the ratio is bounded above by
	\begin{equation}\label{eq:ratio}
		\frac{\Gamma_{j+1}}{\Gamma_j} \le \frac{1}{(j+1)\big((k-j)p-2\big)\cdots\big((k-j)p-p+1\big)} \le \frac{1}{(j+1)(p-2)!}.
	\end{equation}
	It follows that $\Gamma_{j+1}<\Gamma_j$ and that the sum in \eqref{eq:n_2k} is an alternating sum of decreasing positive terms. Call $S_k = p_{2k}\cdot n_{2k} = \sum_{j=0}^k (-1)^j\Gamma_j$, then we have
	\[ \Gamma_0-\Gamma_1 \le S_k \le \Gamma_0 \]
	by taking $j=0$ in \eqref{eq:ratio}.
	Finally,
	\[ \Gamma_0-\Gamma_1 = \Gamma_0\left(1-\frac{\Gamma_1}{\Gamma_0}\right) \ge \frac{\Gamma_0}{(p-2)!}. \]
	Thus
	\[ (kp)!\left(1-\frac{1}{(p-2)!}\right) \le S_k \le (kp)!, \]
	and dividing through by $p^{2k}$ gives the result.
\end{proof}

\begin{remark}
	\begin{enumerate}[label=(\roman*)]
		\item The usual manipulations with alternating sums show $S_k$ is bounded above by $\Gamma_0-\Gamma_1+\cdots+\Gamma_{2b}$ and below by $\Gamma_0-\Gamma_1+\cdots-\Gamma_{2b+1}$ for any $b$.
		
		\item For $k\le a\le 2k$, \eqref{eq:na} gives an explicit formula for the quantity $n_a$. Similar techniques, not developed in detail here, show that the sum in $n_a$ is dominated by the $j=2k-a$ term. Thus, for $k\le a\le 2k$,
		\begin{equation}\label{eq:na-asym}
			n_a \sim \frac{1}{p^a}\big((a-k)p\big)!(2k-a)!\binom{k}{a-k}^2\big(p(p-1)\big)^{2k-a}.
		\end{equation}
		As a check, when $a=k$ (the smallest possible value for $a$), we know $n_k = k!(p-1)^k$, and so \eqref{eq:na} and the computation in \cref{ex:n=p} agree.
		
		\item The value of $n_a$ falls off extremely rapidly from $n_{2k}$. Straightforward asymptotics show a super-exponential decrease: letting $u=a-k$ and $b=2k-a$ and recalling that $k<p$ we observe from \eqref{eq:na-asym} that
		\[ \frac{n_a}{n_{2k}} \sim \big(p^2(p-1)\big)^v v! \binom{k}{u}^2 \frac{(up)!}{(kp)!} 
			\le p^{4v}\cdot 4^p \cdot \frac{(up)!}{(kp)!} 
			= \left( \frac{p^4}{((k-v)p)^{p/2}} \right)^v \cdot \left( \frac{4}{((k-v)p)^{v/2}} \right)^p, \]
		from which we can conclude, for example, that
		\[ \frac{n_a}{n_{2k}} = o(p^{-p/4}). \]
		
	\end{enumerate}
\end{remark}

\bigskip
\section{General case}\label{sec:4}

\Cref{thm:1} considers $f(n,p)=\Prob(P_n\cap P_n^x>1)$ with $P_n$ a Sylow $p$-subgroup of $S_n$. We begin by clarifying what probability is being computed. There are three possibilities:
\begin{enumerate}[(1)]
	\item Pick, uniformly at random and independently, two Sylow $p$-subgroups from the list of all Sylow $p$-subgroups of $S_n$, i.e.~$\Syl_p(S_n)$.
	\item Fix $P_n\in\Syl_p(S_n)$ and pick, uniformly, a Sylow subgroup from $\Syl_p(S_n)$.
	\item Fix $P_n\in\Syl_p(S_n)$, choose $x\in S_n$ uniformly and consider $P_n\cap P_n^x$.
\end{enumerate}
In all cases, $P_n=P_n^x$ is allowed.

\begin{lemma}\label{lem:prob}
	Under any of (1), (2) and (3), the probability $\Prob(P_n\cap P_n^x>1)$ is the same.
\end{lemma}

\begin{proof}
	From Sylow's theorems, $S_n$ acts transitively on $\Syl_p(S_n)$ by conjugation. The uniform distribution is invariant under conjugation, so clearly (1) and (2) give the same probability.
	With $P_n$ fixed, $P_n^{gx}=P_n^x$ if and only if $g\in N_{S_n}(P_n)$. So the uniform distribution on $\Syl_p(S_n)$ assigns probability $\frac{1}{n!}|N_{S_n}(P_n)|$ to each, giving the equality between (2) and (3).
\end{proof}

Throughout, we will use the probabilistic set-up described in (3).

\subsection{Odd primes}
For a prime $p$, let $x\in S_n$ be an element of order $p$ with fewer than $p$ fixed points. Note that this uniquely defines the conjugacy class of $x$. 
Let $f'(n,p)$ be the probability that two random conjugates of $x$, say $x^a$ and $x^b$ with $a,b$ uniform and independent in $S_n$ both centralize a common element of order $p$ (i.e.~the centralizer of $\langle x^a, x^b \rangle$ has order divisible by $p$). As above, it suffices to consider all pairs $x, x^b$ as $b$ ranges uniformly over all $b$.  

We note:

\begin{lemma}\label{lem:1'}
	$f(n,p)  \le f'(n,p)$.
\end{lemma}

\begin{proof}
	Let $P$ and $Q$ be Sylow $p$-subgroups of $S_n$. Let $x \in Z(P)$ have order $p$ and fewer than $p$ fixed points. If $Q=P^y$, then $x^y \in Z(Q)$. If $w \in P \cap Q$ has order $p$, then $w$ commutes with $x$ and $x^y$. Letting $y$ range over all elements, the result follows. 
\end{proof}  

Thus, it suffices to prove the following.           

\begin{theorem}\label{thm:1'}
	If $p\ne 2$, then $\lim_{n\to\infty} f'(n,p)=0$.
\end{theorem}
 
The proof of \cref{thm:1} uses the following special case of results of Eberhard and Garzoni \cite{EG21,EG22}. See in particular \cite[Theorem 1.1]{EG22}.

\begin{theorem}[\cite{EG22}]\label{thm:EG}
	Let $x_n\in S_n$ be an element of odd order for each $n\in\N$ such that the number of fixed points $F(x_n)$ satisfies $F(x_n)/n^{1/2}\to 0$ as $n\to\infty$. Then the probability that two random conjugates of $x_n$ generate $A_n$ goes to $1$ as $n\to\infty$.
\end{theorem}

\begin{remark}
	The probability in \cref{thm:EG} can be taken either as the chance that $\langle x_n^a,x_n^b\rangle=A_n$ with $a,b\in S_n$ chosen uniformly and independently, \textit{or} as the chance that two independent uniformly chosen elements of the class of $x_n$ generate $A_n$. As in \cref{lem:prob}, they agree.
\end{remark}

We also use the following elementary result.

\begin{lemma}\label{lem:elem}
	Let $L$ be a subgroup of $G=S_n$.
	\begin{enumerate}
	\item If $L$ acts primitively, then either $L$ is cyclic of order $n$ with $n$ prime or $C_G(L)=1$.
	\item If $L$ acts transitively and $r$ is a prime dividing $|C_G(L)|$, then $r \mid n$.
	\end{enumerate}
\end{lemma}

\begin{proof}
	Suppose $x \in C_G(L)$ has prime order $r$. Then the orbits of $x$ are permuted by $L$. If $L$ is primitive, there is only one such orbit. If $L$ is transitive, all orbits must have the same size and the result follows.
\end{proof} 

To begin the proof of \cref{thm:1'}, consider $P\in\Syl_p(A_n)$. Clearly $P\in \Syl_p(S_n)$ since $p$ is odd. Let $z\in Z(P)$ be an element with $r\leq p-1$ fixed points where $n=kp+r$. Then the probability that $P\cap P^x \ne 1$, for $x$ uniform in $S_n$, is at most the probability that there exists a common element centralizing both $z$ and $z^x$. By \cref{thm:EG}, this goes to 0 as long as $p/n^{1/2}\to 0$.

If $p$ is fixed, clearly this holds, so it will be assumed that $p$ is increasing. If say $p \log p < n^{1/2}$, the result similarly holds (although we will not use this fact). 
If $p\mid n$, then $r=0$ and the result holds. We use the following result to deduce the result from the case $r=0$.

\begin{lemma}\label{lem:special}
	$f'(n,p) \le \max\{f'(\lambda p, p) \mid 1 \le \lambda \le n/p\}$. 
\end{lemma}

\begin{proof}
	Let $G=S_n$. We partition the conjugacy class $x^G$ of $x$ into disjoint subsets. If $y \in x^G$, let $H = H_y = \langle x, y \rangle$ and $\Omega=\Omega(y)$ denote the union of all $H$-orbits of size prime to $p$. Let $\Delta(\Omega)$ denote the set of conjugates $y$ of $x$ with $\Omega(y)$ a fixed subset. 

	Note that since $x$ has fewer than $p$ fixed points, $H$ has fewer than $p$ orbits on $\Omega$. Suppose $w \in C_G(H)$ has order $p$. Then $w$ acts on $\Omega$. Since $H$ has fewer than $p$ orbits, $w$ acts on each $H$ orbit contained in $\Omega$. By \cref{lem:elem}, $w$ acts trivially on $\Omega$.

	Let $\Omega'$ be the complement of $\Omega$. If $\Omega'$ is empty, then no element of order $p$ centralizes $H$. Let $x'$ and $y'$ denote the restriction of $x$ and $y$ to $\Omega'$ (these are fixed point free permutations on $\Omega'$). Note that as $y$ ranges over $\Delta(\Omega)$, $y'$ ranges uniformly over all fixed point free permutations of order $p$ on $\Omega'$. Since any $p$-element in $C_G(H)$ acts trivially on $\Omega$, it follows that the probability that $x$ and a random element of $\Delta(\Omega)$ are both centralized by an element of order $p$ is precisely $f'(|\Omega'|, p)$.   

	The result follows since $|\Omega'| = \lambda p$ for some $\lambda$. 
\end{proof}

By \cref{thm:EG} it follows that $\lim_{p \rightarrow \infty} f'(\lambda p,p) = 0$ and \cref{thm:1'} follows.  
 
\begin{remark}
	The argument shows that \cref{thm:1} holds uniformly in $2<p\le n$.  Using the ideas and results from \cite{EG22}, one could prove a more refined version of \cref{lem:special}. 
\end{remark}

\subsection{$n$ even, $p=2$}
In this case $Z=Z(P_n)$ is an elementary 2-group. The same holds for $Z(P_n^x)$. Let $z \in Z$ be an involution with no fixed points.   
Consider the partition $\Delta$ of $n$ into $\frac{n}{2}$ disjoint subsets of size $2$ which are the orbits of $z$. Let $A$ be the elementary abelian $2$-subgroup preserving each subset of size $2$. We will show that for a random $g$, the probability that $A \cap A^g \ne 1$ is bounded away from $0$ and so also for $P_n \cap P_n^g$.

Note that $g\Delta$ is a random partition of the same sort. Let $W$ be the number of matching pairs in $\Delta$ and $g \Delta$. Note that $W \ge 1$ if and only if $A \cap A^g \ne 1$.  
The arguments in \cite[\textsection 3]{DHolmes} show that, when $n$ is large, $W$ has a limiting $\operatorname{Poisson}(1/2)$ distribution
\[ \Prob(W=l) \to \frac{e^{-1/2}\cdot(1/2)^l}{l!}\quad\text{as }n\to\infty. \]
In particular,
\[ \Prob(W>0) \sim 1-e^{-1/2}\quad\text{as } n\to\infty, \]
and so the probability that $A \cap A^g \ne 1$ is uniformly bounded away from $0$.

\subsection{$n$ odd, $p =2$}   
We show that the $n$ even case implies the result for $n$ odd. Choose an involution $z \in Z(P_n)$ with exactly $1$ fixed point. Let $z^g$ be a random conjugate of $z$. Note that $\langle z, z^g \rangle$ is a dihedral group that has exactly $1$ orbit of odd size $k$. If $k < n$, then using the result for $n-k$, we see that the probability that $A \cap A^g \ne 1$ (with $A$ the elementary abelian subgroup preserving each orbit of $x$) is bounded uniformly away from $0$. The result now follows by noting that the probability that $zz^g$ is an $n$-cycle goes to $0$ as $n \rightarrow \infty$. 

\bigskip
\section{Double cosets of size $p|P_n|$}\label{sec:new5}
This section gives closed form formulas for the number of $p$-Sylow double cosets of $S_n$ of second smallest size. The main result is the following.

\begin{theorem}\label{thm:2nd-smallest}
	Let $p$ be a prime and $n\in\N$. Suppose the $p$-adic expansion of $n$ is $\sum_{i=0}^\infty a_ip^i$. Let $P_n\in\Syl_p(S_n)$. Then the number of $(P_n,P_n)$-double cosets in $S_n$ of size $p|P_n|$ is 
	\begin{align*}
		&\sum_{i=0}^\infty a_ia_{i+1}+\sum_{i=2}^\infty a_i&&\text{if }p=2,\\
		&\frac{|N_{S_n}(P_n):P_n|}{p}\sum_{i=0}^\infty a_{i+1}\biggl(\binom{p+a_i}{p}(p-1)^{i(p-1)}(p-2)!-1\biggr)&&\text{if }p>2.
	\end{align*}
\end{theorem}

We remark that \cref{thm:2nd-smallest} is proved following a useful lemma for general groups which yields a theorem for the number of maximal size double cosets. The proof also yields an algorithmic description of double cosets of size $p|P_n|$. Recall from \eqref{eq:normaliser} that $|N_{S_n}(P_n):P_n|=\prod_i (p-1)^{ia_i}\cdot a_i!$, when the $p$-adic expansion of $n$ is $\sum_{i=0}^\infty a_ip^i$.

\begin{lemma}\label{lem:Sylow}
	Let $G$ be a finite group, $p$ be a prime and $P\in\Syl_p(G)$. For $k\in\N$ such that $p^k\le|P|$, let 
	\[d(p^k):=\bigl|\{Q\in\Syl_p(G):|P:P\cap Q|=p^k\}\bigr|.\] 
	Then the number of $(P,P)$-double cosets in $G$ of size $|P|p^k$ is 
	\[ \frac{d(p^k)}{p^k}|N_G(P):P|. \]
\end{lemma}

\begin{proof}
	Let $g\in G$ and suppose $|PgP|=|P|p^k$. There exist $x_1,\ldots,x_{p^k}\in P$ such that $PgP=\bigcup_{i=1}^{p^k}x_igP$. Setting $Q_i:=x_igPg^{-1}x_i^{-1}\in\Syl_p(G)$, we obtain $|P:P\cap Q_i|=p^k$ for $i=1,\ldots,p^k$. 
	Suppose that $Q_i=Q_j$. Then 
	\[x_j^{-1}x_i\in N_G(gPg^{-1})\cap P\le gPg^{-1},\] 
	because $gPg^{-1}$ is the unique Sylow $p$-subgroup of $N_G(gPg^{-1})$. Hence, $x_ig\in x_jgP$ and $x_igP=x_jgP$. Consequently, $i=j$. This shows that each $(P,P)$-double coset gives rise to $p^k$ Sylow $p$-subgroups $Q$ of $G$ with $|P:P\cap Q|=p$. 
	Now for each $h\in N_G(P)$ we have 
	\[|PghP|=|PghPh^{-1}g^{-1}|=|PgPg^{-1}|=|PgP|\] 
	and $Q_i=x_ighPh^{-1}g^{-1}x_i^{-1}$. 
	Therefore, each $Q$ arises from $|N_G(P):P|$ double cosets. 
\end{proof}

It is no coincidence that $p^k$ divides $d(p^k)$. In fact, $P\cap Q$ is the stabilizer of $Q$ under the action of $P$ on $\Syl_p(G)$ by conjugation. Hence, $d(p^k)$ is $p^k$ times the number of orbits of size $p^k$. 

For the rest of this section, $G=S_n$.

\begin{lemma}\label{lem:intransitive}
	Let $p$ be a prime, $n\in\N$ and let $P\in\Syl_p(S_{p^n})$. Then $P$ has a unique maximal subgroup of the form $P_1\times\ldots\times P_p$, where $N_1\cup\ldots\cup N_p=\{1,\ldots,p^n\}$, $|N_1|=\ldots=|N_p|=p^{n-1}$ and $P_i\in\Syl_p(\Sym(N_i))$ for all $i\in\{1,\dotsc,p\}$. 
\end{lemma}

\begin{proof}
	It is clear that $P$ has such a subgroup $H=P_1\times\ldots\times P_p$ corresponding to
	\[ N_i:=\{(i-1)p^{n-1}+1, (i-1)p^{n-1}+2, \dotsc, ip^{n-1}\}, \]
	for instance.
	Suppose that there is another such subgroup $Q=Q_1\times\ldots\times Q_p$ corresponding to a partition $M_1\cup\ldots\cup M_p=\{1,\ldots,p^n\}$. Since $P=HQ$, we have that $|H\cap Q|=\frac{1}{p}|H|$. By way of contradiction, suppose that $\{N_1,\ldots,N_p\}\ne\{M_1,\ldots,M_p\}$. Then at least two of the $N_i$ split into smaller subsets of the form $N_i\cap M_j$. Since the subsets $N_i\cap M_j$ are the orbits of $H\cap Q$, this leads to the contradiction $|H\cap Q|<\frac{1}{p}|H|$. 
	Hence, without loss of generality, $N_i=M_i$ for all $i\in\{1,\dotsc,p\}$, and $P_i=P\cap\Sym(N_i)=Q_i$ and $H=Q$.
\end{proof}

\begin{lemma}\label{lem:regsub}
	Let $n\in\N$ with $n\ge 2$, and let $P\in\Syl_2(S_{2^n})$. Then $P$ has unique maximal subgroup of the form $W\wr V$, where $W$ is a Sylow $2$-subgroup of a symmetric group on $2^{n-2}$ letters and $V\cong C_2^2$ acts regularly on the four copies of $W$. 
\end{lemma}

\begin{proof}
	We have $P=W\wr R\cong (W_1\times\ldots\times W_4)\rtimes R$ where $\{1,\ldots,2^n\}=N_1\sqcup\ldots\sqcup N_4$, $W_i\in\Syl_2(\Sym(N_i))$ and $R\cong C_2\wr C_2$ is a Sylow 2-subgroup of $\Sym(\{N_1,\ldots,N_4\})$; note that $W=1$ when $n=2$. Taking $V$ to be the Klein four-subgroup of $R$ gives us one subgroup $H=W\wr V$ of the desired form. 
	
	To prove uniqueness, we use a counting argument. The number of partitions of $\{1,\ldots,2^n\}$ into four subsets of size $2^{n-2}$ each is 
	\[\frac{2^n!}{(2^{n-2}!)^4 4!}.\]
	The number of Sylow $2$-subgroups of $S_{2^k}$ is the $2'$-part of $2^k!$, call it $t_k$. Therefore, the number of subgroups of $S_{2^n}$ of the form $W:=W_1\times\ldots\times W_4$ where each $W_i$ is a Sylow $2$-subgroup of a symmetric group on $2^{n-2}$ letters, is 
	\[\frac{2^n!}{(2^{n-2}!)^4 4!}\cdot t_{n-2}^4=\frac{t_n}{3}.\]
	The Sylow 2-subgroups of $S_{2^n}$ containing $W$ correspond one-to-one to the three Sylow 2-subgroups of $S_4$ (permuting the $W_i$). Since $t_n$ is the number of Sylow $2$-subgroups of $S_{2^n}$, each Sylow $2$-subgroup $P$ of $S_{2^n}$ contains a unique subgroup of the form $W$. Finally, there is only one way to extend $W$ to $W\wr V$ where $V$ permutes the $W_i$ regularly.  
\end{proof}

\begin{proof}[Proof of \cref{thm:2nd-smallest}]
	We write $G=S_n$ and $P=P_n$ for short.
	There exists a partition
	\[\{1,\ldots,n\}=N_1\cup\ldots\cup N_s\]
	and $P_i\in\Syl_p(\Sym(N_i))$ such that $P=P_1\times\ldots\times P_s$, where $s:=\sum_{i=0}^\infty a_i$.
	By \cref{lem:Sylow}, it suffices to determine the number of $Q\in\Syl_p(G)$ such that $|P:P\cap Q|=p$.
	In this case $R:=P\cap Q$ is a radical subgroup of $G$, i.~e. $P$ is the largest normal $p$-subgroup of $N_G(P)$. The radical $p$-subgroups of symmetric groups were classified by Alperin--Fong~\cite[(2A)]{AlperinFong}; see also \cite{Fong} for corrections. 
	
	It turns out that there is a refined partition $M_1\cup\ldots\cup M_t$ of $\{1,\ldots,n\}$ where each $M_i$ is a subset of some $N_j$, and $R_i\le \Sym(M_i)$ such that $R=R_1\times\ldots\times R_t$. Moreover, each $R_i$ is an iterated wreath product of the form $A_1\wr\ldots\wr A_k$, where every $A_j\cong C_p^{e_j}$ (for some $e_j\in\N$) is elementary abelian and acts regularly on its support. (Note the case $k=0$ i.e.~$R_i=1$ is allowed.) Hence, $|M_i|=p^{e_1+\ldots+e_k}$ and
	\[ \log_p|R_i| = e_1p^{e_2+\ldots+e_k} + e_2p^{e_3+\ldots+e_k} + \ldots + e_k \le p^{e_1+\dotsc+e_k-1} + p^{e_2+\dotsc+e_k-1} + \cdots + p^{e_k-1}, \] 
	since $e_i\le p^{e_i-1}$ (with equality if and only if $e_i=1$ or $p=e_i=2$). Thus
	\[ \log_p|R_i| \le \sum_{i=0}^{e_1+\cdots+e_k-1} p^i = \sum_{j=1}^\infty \left\lfloor \frac{|M_i|}{p^j} \right\rfloor = \log_p(|M_i|!). \]
	Since $|P:R|=p$, we conclude that \textit{either} $e_1=\ldots=e_k=1$, \textit{or} $p=2$ and $e_1=\ldots=e_{k-1}=1$ and $e_k=2$. 
	
	Suppose first that $p>2$. Then each $R_i$ is a Sylow $p$-subgroup of the symmetric group on its support. This means that $N_i=M_i$ and $R_i=P_i$ for all but one $i$. 
	For $i\ge 0$, we have $a_{i+1}$ choices to fix a factor $P_j$ of $P$ such that $|N_j|=p^{i+1}$. The only way to decompose $N_j$ into some $M_k$ is to take $p$ disjoint subsets each of size $p^i$, say $N_j=M_{k_1}\cup\ldots\cup M_{k_p}$. 
	By \cref{lem:intransitive}, the $M_{k_l}$ are uniquely determined up to order by the wreath structure of $P_j$. 
	Since $P_r=R_r$ for all $r\ne j$, we have $a_i+p$ sets $M_l$ of size $p^i$ in total. Each Sylow $p$-subgroup $Q$ containing $R$ combines $p$ of those $M_l$ to one set of size $p^{i+1}$. We have $\binom{p+a_i}{p}$ possibilities to choose those sets $M_l$. For ease of notation, suppose that $M=M_1\cup\ldots\cup M_p$ have been chosen. Now we need to count how many Sylow $p$-subgroups of $S_M:=\Sym(M)$ contain $R_M:=R_1\times\ldots\times R_p$ with $R_i\in\Syl_p(\Sym(M_i))$. 
	Each such Sylow $p$-subgroup $P_M$ lies inside $N_{S_M}(R_M)\cong (R_1\rtimes C_{p-1}^i)\wr S_p$. On the other hand, $N_{S_M}(P_M)\cong P_M\rtimes C_{p-1}^{i+1}$. Hence, the number of Sylow $p$-subgroups of $S_M$ containing $R_M$ is 
	\[ |N_{S_M}(R_M):N_{S_M}(P_M)|=\frac{(p-1)^{ip}p!}{(p-1)^{i+1}p}=(p-1)^{i(p-1)}(p-2)!. \]
	For each choice of $N_j$ there is just one refined partition $M=M_1\cup\ldots\cup M_p$ leading to $Q=P$. This possibility needs to be subtracted. This proves the theorem for $p>2$.
	
	Finally, let $p=2$. The subgroups $Q$ constructed above also exist here. The corresponding number simplifies to $\sum a_ia_{i+1}$ since $a_i\in\{0,1\}$. Let $i\ge 2$ be such that $a_i=1$. Without loss of generality, let $|N_i|=2^i$. Let $P_i\le\Sym(N_i)$ be the corresponding factor of $P$. By \cref{lem:regsub}, $P_i$ has a unique maximal subgroup $R_i\cong W\wr V=W^4\rtimes V$ such that $W$ is a Sylow 2-subgroup of some $S_{2^{i-2}}$ and $V\cong C_2^2$ permutes the conjugates of $W$ regularly. Since $N_{\Sym(N_i)}(R_i)\cong W\rtimes S_4$, the Sylow 2-subgroups $Q$ containing $R_i$ correspond one-to-one to the three Sylow 2-subgroups of $S_4$. One of them equals $P$. So for each $i$ with $a_i=1$, we obtain one (additional) double coset of size $2|P|$. This yields the second sum $\sum_{i\ge 2}a_i$ in the formula. 
\end{proof}

From the proofs of \cref{lem:Sylow} and \cref{thm:2nd-smallest}, one can extract an algorithm to construct the double cosets of size $p|P_n|$. 

\bigskip
\section{Remarks and problems}\label{sec:5}
Our original goal was to ``understand the Sylow-$p$ double cosets of $S_n$". We approach this by counting; how many double cosets are there and what are their sizes? 
There is still much that is not known. For example, when $p=2$ then \cref{thm:1}(b) shows that $f(n,2)$ is bounded away from 0. We believe in fact that equality holds, although we do not know the limiting distribution for $|P_n\cap P_n^x|$.
It is further natural to ask:
\begin{itemize}
	\item Are there `nice labels' for the double cosets?
	\item What are the structure constants when multiplying double cosets?
	\item Over a field of characteristic $p$, the Hecke algebra
	\[ \cH_n(k) = L_k(P_n\setminus S_n/P_n) = \End_{S_n}(\triv_{P_n}^{S_n}) \]
	is not semisimple. As mentioned in the introduction, when $p$ is odd, we know from \cite[Corollary B]{GL1} that the number of irreducible representations of $\cH_n(k)$ is equal to $|\cP(n)|$, the number of partitions of $n$, whenever $n$ is not a power of $p$ (barring small exceptions when $p=3$ and $n\le 10$). Conversely, if $n=p^\ell$ for some $\ell\in \N$ then $\cH_n(k)$ admits exactly $|\mathcal{P}(n)|-2$ irreducible representations. 
	
	Moreover, in a recent (unpublished) manuscript, Giannelli and Law were able to completely describe those irreducible characters $\chi$ of $S_n$ such that $[\chi, (\triv_{P_n})^{S_n}]=1$. This result, together with \cite[Chapter 11D]{CR81}, allows us to determine the exact number of $1$-dimensional representations of $\cH_n(k)$ and therefore to compute the $k$-dimension of its abelianisation. 
\end{itemize}
But, can we say more?
Does $\cH_n(k)$ have some kind of nice structure? For instance, is it Frobenius or quasi-Frobenius? (See \cite{CR62} for background on quasi-Frobenius algebras.)
For the roughly parallel problems for finite groups $G$ with a split $BN$-pair, there \textit{are} nice answers to these additional questions.
\begin{itemize}
	\item The Yokonuma algebra \cite{Y} gives a nice description of $U\setminus G/U$, closely tied to Bruhat decompositions.
	\item At least in type $A$ and elsewhere \cite{KMS}, there are useful descriptions of the structure constants.
	\item The work of Tinberg \cite{Tinberg} summarised in the introduction shows that $L_k(U\setminus G/U)$ \textit{is} Frobenius (among many other things).
\end{itemize}
All we want is to let $q\to 1$ in type $A$(!). Moreover, work of Kessar--Linckelmann suggests that there is trouble even in the simplest case $C_p\le S_p$. They find $L_k(C_p\setminus S_p/C_p)$ is Frobenius for $p\in\{3,5\}$ but not quasi-Frobenius for $p\ge7$ \cite{KL}. This paper develops general tools for studying properties such as self-injectivity for Hecke algebras for general groups over general fields.

We would like a more quantitative version of \cref{thm:1}, along the lines of \cref{sec:4}. We would like analogues of the counting theorems for $n_a$ for the case of general $n$.

\bigskip


\begin{thebibliography}{9999}
	\bibitem[AP90]{AlperinFong}
	{\sc J.~L.~Alperin and P.~Fong,}
	\newblock Weights for symmetric and general linear groups,
	\newblock \textit{J.~Algebra} \textbf{131} (1990), 2--22.
	
	\bibitem[BD24]{BD}
	{\sc L.~Bartholdi and P.~Diaconis,}
	\newblock An algorithm for uniform generation of unlabeled trees (P\'olya trees), with an extension of Cayley's formula,
	\newblock \texttt{arXiv:2411.17613 [math.CO]}.
	
	\bibitem[Car72]{Carter72}
	{\sc R.~W.~Carter,}
	\newblock \textit{Simple groups of Lie type,}
	\newblock Pure and applied mathematics \textbf{28}, Wiley, 1972.
	
	\bibitem[Car85]{Carter85}
	{\sc R.~W.~Carter,}
	\newblock \textit{Finite groups of Lie type,}
	\newblock Wiley, 1985.

	\bibitem[Cur70]{Curtis70}
	{\sc C.~W.~Curtis,}
	\newblock \textit{Modular representations of finite groups with split $(B,N)$-pairs,}
	\newblock Lecture Notes in Mathematics \textbf{131}, Springer, 1970.
	
	\bibitem[Cur99]{Curtis99}
	{\sc C.~W.~Curtis,}
	\newblock \textit{Pioneers of representation theory: Frobenius, Burnside, Schur, Brauer,}
	\newblock History of Mathematics \textbf{15}, American Mathematical Society, Providence, 1999.
	
	\bibitem[CR62]{CR62}
	{\sc C.~W.~Curtis and I.~Reiner,}
	\newblock \textit{Representation theory of finite groups and associative algebras,}
	\newblock Interscience, New York, 1962.
	
	\bibitem[CR81]{CR81}
	{\sc C.~W.~Curtis and I.~Reiner, }
	\newblock \emph{Methods of representation theory: with applications to finite groups and orders, vol.~1},
	\newblock John Wiley \& Sons, New York, 1981.
	
	\bibitem[DHol02]{DHolmes}
	{\sc P.~Diaconis and S.~Holmes,}
	\newblock Random walks on trees and matchings,
	\newblock \textit{Electron.~J.~Probab.} \textbf{7} (2002), 1--17.
	
	\bibitem[DHow25]{DHowes}
	{\sc P.~Diaconis and M.~Howes,}
	\newblock Random sampling of contingency tables and partitions: two practical applications of the Burnside process,
	\newblock \textit{Technical report}, Department of Statistics, Stanford University.
	
	\bibitem[DMal21]{DMalliaris}
	{\sc P.~Diaconis and M.~Malliaris,}
	\newblock Complexity and randomness in the Heisenberg groups (and beyond). 
	\newblock \textit{N.~Z.~J.~Math.} \textbf{52} (2021), 403--426.
	
	\bibitem[DRS23]{DRS}
	{\sc P.~Diaconis, A.~Ram and M.~Simper,}
	\newblock Double coset Markov chains,
	\newblock \textit{Forum Math.~Sigma} \textbf{11}:e2 (2023), 1--45.
	
	\bibitem[DSi22]{DSimper}
	{\sc P.~Diaconis and M.~Simper,}
	\newblock Statistical enumeration of groups by double cosets,
	\newblock \textit{J.~Algebra} \textbf{607} (2022), 214--246.
	
	
	\bibitem[DZ23]{DZ20}
	{\sc P.~Diaconis and Z.~Zhong,}
	\newblock Hahn polynomials and the Burnside process,
	\newblock \textit{Ramanujan J.} \textbf{61} (2023), 567--595.
	
	\bibitem[DZ25]{DZ25}
	{\sc P.~Diaconis and Z.~Zhong,}
	\newblock Counting the number of group orbits by marrying the Burnside process with importance sampling,
	\newblock \texttt{arXiv:2507.11731 [math.PR]}.
	
	\bibitem[EG21]{EG21}
	{\sc S.~Eberhard and D.~Garzoni,}
	\newblock Random generation with cycle type restrictions,
	\newblock \textit{Algebr.~Comb.} \textbf{4}(1) (2021), 1--25.
	
	\bibitem[EG22]{EG22}
	{\sc S.~Eberhard and D.~Garzoni,}
	\newblock Probability of generation by random permutations of given cycle type,
	\newblock \texttt{arXiv:2205.07573 [math.GR]}.
	
	\bibitem[F20]{Fong}
	{\sc P.~Fong,}
	\newblock The Alperin weight conjecture for symmetric and general linear groups revisited,
	\newblock \textit{J.~Algebra} \textbf{558} (2020), 395--410.
	
	\bibitem[GAP]{GAP48}
	{\sc The GAP~Group,} 
	\newblock \textit{GAP -- Groups, Algorithms, and Programming, Version 4.14.0},
	\newblock 2024, (\url{http://www.gap-system.org}).

	\bibitem[Gia21]{Giannelli}
	{\sc E.~Giannelli,}
	\newblock McKay bijections for symmetric and alternating groups,
	\newblock \textit{Algebra Number Theory} \textbf{15}(7) (2021), 1809--1835.
	
	\bibitem[GL18]{GL1}
	{\sc E.~Giannelli and S.~Law,}
	\newblock On permutation characters and Sylow $p$-subgroups of $\mathfrak{S}_n$,
	\newblock \emph{J.~Algebra} \textbf{506} (2018), 409--428.
	
	\bibitem[GJ02]{GJ}
	{\sc L.~Goldberg and M.~Jerrum,}
	\newblock The Burnside process converges slowly,
	\newblock \textit{Comb.~Probab.~Comput.} \textbf{11}(1) (2002), 21--34.
	
	\bibitem[GO96]{GO}
	{\sc A.~Granville and K.~Ono,}
	\newblock Defect zero $p$-blocks for finite simple groups, 
	\newblock \emph{Trans.~Amer.~Math.~Soc.} \textbf{348} (1996), 331--347. 
	
	\bibitem[Gre78]{Green}
	{\sc J.~A.~Green,}
	\newblock On a theorem of H.~Sawada,
	\newblock \textit{J.~London~Math.~Soc.} (2) \textbf{18} (1978), 247--252.
	

	\bibitem[JK81]{JK}
	{\sc  G.~James and A.~Kerber,}
	\newblock \emph{The representation theory of the symmetric group}, 
	\newblock Encyclopedia of Mathematics and its Applications, vol.~16, Addison-Wesley Publishing Co., Reading, Mass., 1981.
	
	\bibitem[J93]{Jerrum}
	{\sc M.~Jerrum,}
	\newblock Uniform Sampling Modulo a Group of Symmetries Using Markov Chain Simulation,
	\newblock in J.~Friedman (ed.) \textit{Expanding Graphs, Proceedings of a DIMACS Workshop}, DIMACS Series in Discrete Mathematics and Theoretical Computer Science \textbf{10}, DIMACS/AMS, 1993 (pp.~37--47).

	\bibitem[Kal48]{Kaloujnine}
	{\sc L.~Kaloujnine,}
	\newblock La structure des $p$-groupes de Sylow des groupes sym\'etriques finis,
	\newblock \textit{Ann.~Sci.~\'Ecole Norm.~Sup.} (3) \textbf{65} (1948), 239--276.
	
	\bibitem[Kel03]{Keller}
	{\sc T.~Keller,}
	\newblock Orbits in finite group actions,
	\newblock in C.~M.~Campbell, E.~F.~Robertson, G.~C.~Smith (eds.) \textit{Groups St Andrews 2001 in Oxford}, London Mathematical Society Lecture Note Series \textbf{305}, Cambridge University Press, 2003 (pp.~306--331).
	
	\bibitem[Kle05]{Kleshchev}
	{\sc A.~Kleshchev,}
	\newblock \textit{Linear and projective representations of symmetric groups,}
	\newblock Cambridge Tracts in Mathematics, Cambridge University Press, 2005.
	
	\bibitem[KL25]{KL}
	{\sc R.~Kessar and M.~Linckelmann,}
	\newblock On some Hecke algebras for $p$-blocks of finite groups algebras,
	\newblock \textit{preprint}.
	
	\bibitem[KMS23]{KMS}
	{\sc R.~Kessar, G.~Malle and J.~Semeraro,}
	\newblock The principal block of a $\Z_l$-spets and Yokonuma type algebras,
	\newblock \textit{Algebra Number Theory} \textbf{17}:2 (2023), 397--433.
	
	
	\bibitem[L00]{Leg}
	{\sc A.~Legendre,}
	\newblock \textit{Th\'eorie des Nombres,}
	\newblock Firmin Didot Fr\`eres, fourth edition, 1900.
	
	\bibitem[M04]{Meo}
	{\sc M.~Meo,}
	\newblock The mathematical life of Cauchy's group theorem,
	\newblock \textit{Hist.~Math.} \textbf{31} (2) (2004), 196--221.
	
	
	\bibitem[N98]{NavarroBook}
	{\sc G.~Navarro,}
	\newblock \textit{Characters and Blocks of Finite Groups,}
	\newblock London Mathematical Society Lecture Note Series \textbf{250}, Cambridge University Press, 1998.
	
	\bibitem[OEIS23]{OEIS}
	{\sc R.~Stanley,}
	\newblock A360808, Number of double cosets of the Sylow 2-subgroup of the symmetric group $S_n$, 2023, 
	\newblock \url{https://oeis.org/A360808}.
	
	\bibitem[O94]{OlssonBook}
	{\sc J.~B.~Olsson,}
	\newblock \emph{Combinatorics and representations of finite groups},
	\newblock Vorlesungen aus dem Fachbereich Mathematik der Universit\"at Essen, Heft 20, 1994.
	
	\bibitem[Ren24]{Renteln}
	{\sc P.~Renteln,}
	\newblock Counting Sylow double cosets in the symmetric group,
	\newblock \textit{Discrete Math.} \textbf{347} (10) (2024), 114109.
	
	\bibitem[Ric69]{Richen}
	{\sc F.~Richen,}
	\newblock Modular representations of split $(B,N)$-pairs,
	\newblock \textit{Trans.~Am.~Math.~Soc.} \textbf{140} (1969), 435--460.
	
	\bibitem[Saw77]{Sawada}
	{\sc H.~Sawada,}
	\newblock A characterization of modular representations of finite groups with split $(B,N)$-pairs,
	\newblock \textit{Math.~Z.} \textbf{155} (1977), 29--42.
	
	\bibitem[Sta99]{Stanley}
	{\sc R.~P.~Stanley,}
	\newblock \textit{Enumerative Combinatorics, volume 2,}
	\newblock Cambridge Studies in Advanced Mathematics \textbf{62}, Cambridge University Press, 1999.
	
	\bibitem[Suz82]{Suzuki}
	{\sc M.~Suzuki,}
	\newblock \textit{Group theory I,}
	\newblock Springer Berlin, Heidelberg, 1982.
	
	\bibitem[T80]{Tinberg}
	{\sc N.~Tinberg,}
	\newblock Modular representations of finite groups with unsaturated split $(B,N)$-pairs,
	\newblock \textit{Can.~J.~Math.} \textbf{32}(3) (1980), 714--733.
	
	\bibitem[W16]{Wildon}
	{\sc M.~Wildon,}
	\newblock Sylow subgroups of symmetric groups, 2016, 
	\newblock \url{https://wildonblog.wordpress.com/2016/11/27/sylow-subgroups-of-symmetric-groups/}.
		
	\bibitem[Y67]{Y}
	{\sc T.~Yokonuma,}
	\newblock Sur la structure des anneaux de Hecke d'un groupe de Chevalley fini,
	\newblock \textit{C.~R.~Acad.~Sci.~Paris Ser.~I Math.} \textbf{264} (1967), 344--347.
	
	\bibitem[ZM96]{ZM}
	{\sc V.~I.~Zenkov and V.~D.~Mazurov,}
	\newblock On the intersection of Sylow subgroups in finite groups, 
	\newblock \textit{Algebra i Logika} \textbf{35}(4) (1996), 424--432.
\end{thebibliography}
\end{document}